\documentclass[11pt]{amsart}
\usepackage{palatino}
\usepackage{amsfonts}
\usepackage{amssymb}
\usepackage{amscd}
\usepackage[breaklinks,bookmarksopen,bookmarksnumbered]{hyperref}
\usepackage[dvips]{graphicx}

\newtheorem{theorem}{Theorem}[section]
\newtheorem{proposition}[theorem]{Proposition}

\newtheorem{lemma}[theorem]{Lemma}
\theoremstyle{definition}

\newtheorem{remark}[theorem]{Remark}

\newtheorem{conjecture}[theorem]{Conjecture}
\newtheorem{conjecture/question}[theorem]{Conjecture/Question}

\newtheorem{remark/definition}[theorem]{Remark/Definition}
\newtheorem{terminology/notation}[theorem]{Terminology/Notation}

\def\PP{{\textbf P}}
\def\OO{\mathcal{O}}

\def\cD{\mathcal{D}}

\def\cM{\mathcal{M}}

\def\Pic0{{\rm Pic}^0(X)}

\def\mm{\overline{\mathcal{M}}}

\def\dd{\overline{\mathcal{D}}}

\begin{document}

\title{Brill-Noether with ramification at unassigned points}

\author[G. Farkas]{Gavril Farkas}

\address{Humboldt-Universit\"at zu Berlin, Institut F\"ur Mathematik,  Unter den Linden 6
\hfill \newline\texttt{}
\indent 10099 Berlin, Germany} \email{{\tt farkas@math.hu-berlin.de}}

\thanks{}

\begin{abstract} We discuss how via limit linear series and standard facts about divisors on moduli spaces of pointed curves, one can establish a non-existence Brill-Noether results for linear series with prescribed ramification at unassigned points.

\end{abstract}

\maketitle

In the course of developing their theory of limit linear series, among many other applications, Eisenbud and Harris \cite{EH}, \cite{EH1} also considered
the Brill-Noether problem with \emph{prescribed} ramification at \emph{assigned} points. For a smooth curve $C$ of genus $g$, a point $p\in C$ and a linear series $\ell=(L, V)\in G^r_d(C)$, one denotes by $$\alpha^{\ell}(p):0\leq \alpha_0^{\ell}(p)\leq \ldots \leq \alpha_r^{\ell}(p)\leq d-r$$
the \emph{ramification} sequence of $\ell$ at $p$.

Having fixed points $p_1, \ldots, p_n\in C$, Schubert indices $\bar{\alpha}^j:0\leq \alpha_0^j\leq \ldots \leq \alpha_r^j\leq d-r$ of type $(r, d)$ for $j=1, \ldots, n$, the locus of linear series on $C$ having prescribed ramification at $p_1, \ldots, p_n\in C$, that is,
$$G^r_d\bigl(C, (p_j, \bar{\alpha}^j)\bigr):=\{\ell\in G^r_d(C): \alpha^{\ell}(p_j)\geq \bar{\alpha}^j\  \mbox{ for } j=1, \ldots, n\}$$ is a generalized determinantal variety of expected dimension $$\rho(g, r, d, \bar{\alpha}^1, \ldots, \bar{\alpha}^n):=\rho(g, r, d)-\sum_{j=1}^n\sum_{i=0}^r \alpha_i^j,$$ where $\rho(g, r, d):=g-(r+1)(g-d+r)$ is the Brill-Noether number.  It is proved in \cite{EH1} Theorem 1.1, that for a general pointed curve $[C, p_1, \ldots, p_n]\in \cM_{g, n}$, each component of $G^r_d\bigl(C, (p_j, \bar{\alpha}^j)\bigr)$ has dimension $\rho(g, r, d, \bar{\alpha}^1, \ldots, \bar{\alpha}^n)$. For $n=1$, a necessary and sufficient condition for existence is given.  Denoting the positive part of an integer $n\in \mathbb Z$ by $(n)_+:=\mbox{max}\{n, 0\}$, the pointed curve $[C, p]\in \cM_{g, 1}$ carries a linear series $\ell\in G^r_d(C)$ with ramification $\alpha^{\ell}(p)\geq \bar{\alpha}$  if and only if

\begin{equation}\label{assigned}
\sum_{i=0}^r (\alpha_i+g-d+r)_+\leq g.
\end{equation}
Setting $\bar{\alpha}=(0, \ldots, 0)$, one recovers the Brill-Noether theorem. In this note we explain how the methods of \cite{EH1}, \cite{EH2} coupled with basic facts about  $\mbox{Pic}(\mm_{g, n})$  provide a solution to the Brill-Noether problem with \emph{prescribed} ramification at \emph{unassigned} points.

\begin{theorem}\label{nonexistence}
Let $C$ be a general curve of genus $g\geq 2$\   and \ $\bar{\alpha}:0\leq \alpha_0\leq \ldots \leq \alpha_r\leq d-r$ß  a Schubert ramification index.
If $$\rho(g, r, d)-\sum_{i=0}^r \alpha_i<-1,$$ then $C$  carries no linear series $\ell\in G^r_d(C)$ having ramification $\alpha^{\ell}(p)\geq \bar{\alpha}$ at $\mathrm{some}$ point $p\in C$.
\end{theorem}

As expected, since on an elliptic curve,  due to the existence of translations, one cannot speak of an unassigned point, Theorem \ref{nonexistence} is true for \emph{every} smooth curve of genus $1$. The inequality $\rho(1, r, d, \alpha^{\ell}(p))\geq 0$ holds for every elliptic curve $[E, p]\in \cM_{1,1}$ and every linear series $\ell\in G^r_d(E)$, and lies at the heart of the proof in \cite{EH} of the non-existence part of the Brill-Noether theorem by degeneration to a flag curve with $g$ elliptic tails.  We note that some partial results in the direction of Theorem \ref{nonexistence} were obtained in \cite{L}. The previous result can be, to some extent, generalized to multiple points:

\begin{theorem}\label{multiple}
We fix integers $r, d\geq 1$ and $g\geq 2$ as well as Schubert indices $\bar{\alpha}^1, \ldots, \bar{\alpha}^n$, with
$$\rho(g, r, d, \bar{\alpha}^1, \ldots, \bar{\alpha}^n)<-1.$$
Then the locus $\Bigl\{[C, p_1, \ldots, p_n]\in \cM_{g, n}: G^r_d(C, (p_j, \bar{\alpha}^j))\neq \emptyset \Bigr\}$ has codimension at least two in moduli.
\end{theorem}

Setting $n=1$, one recovers Theorem \ref{nonexistence}. Contrary to our first result, Theorem \ref{multiple} is not true in genus $1$ and fails already on $\cM_{1, 2}$.
Indeed, for an integer $a\geq 2$, we fix an elliptic curve $E$ and distinct points $p_1, p_2\in E$ such that $p_1-p_2\in \mbox{Pic}^0(E)[a]$. Choose a rational function $f\in \mathbb C(E)$ with $\mbox{div}(f)=a\cdot p_2-a\cdot p_1$ and consider the linear series $\ell:=(\OO_E(2a\cdot p_1), V)\in G^2_{2a}(E)$, where $V:=\langle 1, f, f^2\rangle$. Clearly $\alpha^{\ell}(p_j)=(0, a-1, 2a-2)$ for $j=1, 2$, hence one calculates $\rho(1, 2, 2a, \alpha^{\ell}(p_1), \alpha^{\ell}(p_2))=-2$. On the other, the locus $D_a:=\bigl\{[E, p_1, p_2]\in\cM_{1, 2}: p_1-p_2\in \mbox{Pic}^0(E)[a]\bigr\}$ is obviously a divisor in the moduli space. Considering higher dimensional linear series, the same idea produce counterexamples with even more negative Brill-Noether number.

\vskip 3pt

Straightforward generalizations of Theorem \ref{multiple} to lower Brill-Noether numbers and higher codimension in moduli are false when $r\geq 2$, as the following example
illustrates. Let $C$ be a curve of genus $2$ and $L=K_C^{\otimes 2}\in G^2_4(C)$. If $p\in C$ is a Weierstrass point,
then $\alpha^L(p)=(0, 1, 2)$. In particular, if both $p_1, p_2\in C$ are Weierstrass points, then
$\rho(2, 2, 4, \alpha^L(p_1), \alpha^L(p_2))=-4$. On the other hand, the codimension of the locus
$$\bigl\{[C, p_1, p_2]\in \cM_{2, 2}: G^2_4\bigl(C, (p_j, (0, 1, 2))\bigr)\neq \emptyset\bigr\}$$
is two, in particular it projects onto $\cM_2$. Similar examples can be constructed for arbitrary
genus and, in this sense, Theorem \ref{multiple} is optimal. Even though these counterexamples invite caution before formulating new predictions, we believe the following statement should be true:

\begin{conjecture}
Let $g, r, d$ be positive integers and $\bar{\alpha}$ a Schubert index of type $(r, d)$ with $\rho(g, r, d, \bar{\alpha})<-2$. Then the locus of curves
$[C]\in \cM_g$ carrying a linear series $\ell\in G^r_d(C)$ with $\alpha^{\ell}(p)\geq \bar{\alpha}$ for \emph{some} point $p\in C$, has codimension at least two in moduli.
\end{conjecture}

I would like to thank the referee for a careful reading and pertinent comments which improved the presentation.

%Next we discuss existence of linear series with ramification at an unassigned point.

%Fixing a Schubert index $\bar{\alpha}$ such that $\rho(g, r, d, \bar{\alpha})=-1$, note that it is not the case that each curve $[C]\in \cM_g$ has a $\mathfrak g^r_d$ with ramification %$\alpha$ at some point. For instance, let $g=3, d=6, r=2$ and $\bar{\alpha}=(0, 3, 4)$. Then $\rho(g, r, d, \bar{\alpha})=-1$, but a general curve $C$ of genus $3$ carries no linear series %$\ell\in G^2_6(C, \bar{\alpha})$, for else, $h^0(C, \OO_C(2p))=2$ and $C$ is hyperelliptic. By reducing to \cite{EH1}, we prove the following  existence result:

%\begin{theorem}\label{existence}
%We fix a Schubert index $\bar{\alpha}:0\leq \alpha_0\leq \ldots \leq \alpha_r\leq d-r$ such that $\rho(g,r, d, \bar{\alpha})=-1$. If $\alpha_0\geq g-d+r$, then $\mathrm{every}$ curve $C$ of %genus $g$ carries a linear series $\ell\in G^r_d(C)$ with $\alpha^{\ell}(p)\geq \bar{\alpha}$ for some point $p\in C$.
%\end{theorem}
%
%Observe that the condition $\alpha_0\geq g-d+r$ is not satisfied in the previous example.

\section{The proofs}

The proof of Theorem \ref{nonexistence} is an adaptation of an idea that appears in \cite{EH2}. First, by explicit methods, we prove Theorem \ref{nonexistence} in genus $3$. Then, assuming by contradiction that a general curve $[C]\in \cM_g$ possesses a point $p\in C$ with $G^r_d(C, (p, \bar{\alpha}))\neq \emptyset$, the locus of such pairs $[C, p]$ in the universal curve $\cM_{g, 1}$ is an effective divisor. By using the solution to the problem in genus $3$ and basic facts about the Picard group of $\mm_{g, 1}$ (that have already been used in \cite{EH1}, \cite{EH2}), we show that this effective divisor must be empty, thus establishing Theorem \ref{nonexistence}. Theorem \ref{multiple} is proved by induction and relies on the already established case $n=1$.

\vskip 3pt

We work throughout over the complex numbers and we assume familiarity with the theory of limit linear series \cite{EH}. We begin with some preliminaries on moduli spaces of curves and recall that for $g\geq 3$ and $n\geq 1$, the rational Picard group $\mbox{Pic}(\mm_{g, n})$ is freely generated by the Hodge class $\lambda$, the relative cotangent classes $\psi_1, \ldots, \psi_n$, the boundary divisor class $\delta_{\mathrm{irr}}:=[\Delta_{\mathrm{irr}}]$ of irreducible $n$-pointed stable
curves of curves $g$ and by the classes $\delta_{i:S}:=[\Delta_{i:S}]$, where for each $i\geq 0$ and $S\subset \{1, \ldots, n\}$, the general point of the boundary divisor $\Delta_{i:S}$ corresponds to a transverse union of two smooth curves of genus $i$ and $g-i$ respectively, meeting in one point,  the marked points lying on the genus $i$ component being precisely those labeled by $S$. Obviously $\delta_{i:S}=\delta_{g-i:S^c}$. When $n=0$, we write as usually $\delta_i:=\delta_{i:\emptyset}\in \mbox{Pic}(\mm_g)$ for $i=0, \ldots, \lfloor \frac{g}{2} \rfloor$. Similarly, for $n=1$, we write $\delta_i:=\delta_{i:\{1\}}\in \mbox{Pic}(\mm_{g,1})$, for $i=1, \ldots, g-1$. The following well-known fact will be used in the course of the proof of Theorem \ref{multiple}:

\begin{lemma}\label{zerodiv} Let $g\geq 3$ and $\cD$ an effective divisor on $\cM_{g, n}$ such that all the $\lambda, \psi_1, \ldots, \psi_n$ and $\{\delta_{0:S}\}_{|S|=2}$-coefficients in the expansion of the class $[\dd]\in \mathrm{Pic}({\mm_{g, n}})$ of its closure in $\mm_{g, n}$ are  equal to zero. Then $\cD=0$.
\end{lemma}

\begin{proof} Let $\pi:\mm_{g, n}\rightarrow \mm_g$ be the forgetful morphism. If $\pi(\dd)=\mm_g$, then there exists an index $1\leq i\leq n$ such that $\dd$ intersects non-trivially the general fibre $F_i$ of the morphism $\pi_i:\mm_{g, n}\rightarrow \mm_{g, n-1}$ forgetting the $i$-th marked point. Note that $F_i\cdot \psi_i=2g-3+n$, as well as $F_i\cdot \psi_j=1$ and $F_i\cdot \delta_{0:ji}=1$ for $j\in \{i\}^c$, whereas the intersection numbers with all other generators of $\mbox{Pic}(\mm_{g, n})$ are zero. We obtain that $F_i\cdot \dd=0$, a contradiction. Therefore, $\cD=\pi^*(\cD')$, where $\cD'$ is an effective divisor on $\cM_g$. Since the Hodge class on $\mm_{g, n}$ is pulled back via $\pi$ via the Hodge class on $\mm_g$ and the map $\pi^*:\mbox{Pic}(\mm_g)\rightarrow \mbox{Pic}(\mm_{g, n})$ is injective, it follows that the $\lambda$-coefficient in the expansion of $[\dd']\in \mbox{Pic}(\mm_g)$ is equal to $0$ as well. But $\cM_g$ admits a compactification, namely the \emph{Satake compactification} $\mm_g^s$,  having boundary $\mm_g^s-\cM_g$ of codimension $2$; in fact, there is a regular Torelli map $t:\mm_g\rightarrow \mm_g^s$, assigning to a stable curve $[C]\in \mm_g$ the product of the degree zero Jacobian varieties of the components of its normalization $\widetilde{C}$ of $C$, and under this map $t_*(\Delta_i)=0$, for $i=0, \ldots, \lfloor \frac{g}{2}\rfloor$. A consequence of the existence of $\mm_g^s$ is that there exists no effective divisor $\cD'$ on $\cM_g$ whose $\lambda$-coefficient is zero, hence $\cD'=0$, and thus $\cD=0$.
\end{proof}

\vskip 3pt

Following \cite{EH1}, we consider the clutching map $\varphi_{g, n}:\mm_{0, g+n}\rightarrow \mm_{g, n}$, given by
$$\varphi_{g, n}\bigl([R, x_1, \ldots, x_g, p_1, \ldots, p_n]\bigr):=[R \cup_{x_1} E_1 \cup \ldots \cup_{x_g} E_g, \ p_1, \ldots, p_n],$$
obtained by gluing fixed elliptic curves $E_1, \ldots, E_g$ at the first $g$ marked points of a curve $R$ of arithmetic genus $0$. The marked points $p_1, \ldots, p_n$ of the resulting genus $g$ stable curve lie on the rational spine $R$.  We consider the action of the symmetric group $\mathfrak{S}_g$ on $\mm_{0, g+n}$ by permuting the marked points labeled by $x_1, \ldots, x_g$. For a subset $S\subset \{p_1, \ldots, p_n\}$ and an integer   $0\leq i\leq g$ with $2\leq |S|+i\leq g+n-2$, we define the $\mathfrak S_g$-invariant boundary divisor
$$B^S_i:=\sum_{\substack{T\subset \{x_1, \ldots, x_g\}\\ |T|=i}} \delta_{0: S\cup T}\in \mathrm{Pic}(\mm_{0, g+n})^{\mathfrak{S}_g}.$$
Clearly $\varphi_{g, n}^*(\delta_{i:S})=B_i^S$, whenever $i+|S|\geq 2$, whereas $\varphi_{g, n}^*(\psi_i)=\psi_{p_i}$. In order to distinguish between cotangent classes on $\mm_{0, g+n}$ and $\mm_{g, n}$, we label the corresponding marked point by $p_i$ on the genus zero curve and by $i\in \{1, \ldots, n\}$ on the genus $g$ curve.

\begin{lemma}\label{invpic}
The $\mathfrak{S}_g$-invariant rational Picard group $\mathrm{Pic}(\mm_{0, g+2})^{\mathfrak{S}_g}$ is freely generated by the boundary classes $\{B^{p_1}_i\}_{i=1}^{g-1}$ and $\{B^{p_1 p_2}_j\}_{j=1}^{g-2}$. In particular, $\mathrm{dim}_{\mathbb Q}\ \mathrm{Pic}(\mm_{0, g+2})^{\mathfrak{S}_g}=2g-3$.
\end{lemma}

\begin{proof} This follows for instance from \cite{FG} Proposition 1. It is shown in \emph{loc. cit.} that there is precisely one relation between the $\mathfrak{S}_g$-invariant boundary divisor classes, which expresses $B^{p_1 p_2}_0=\delta_{0: p_1p_2}$ in terms of all the other invariant boundary classes. The cotangent classes are also expressible in this basis, see \cite{FG} Lemma 1:
$$\psi_{p_1}=\sum_{j=1}^{g-1}\frac{(g+1-j)(g-j)}{(g+1)g}\bigl(B_j^{p_1}+B_{j-1}^{p_1p_2}\bigr).$$
\end{proof}

\noindent We  fix a general pointed curve $[C, p', q]\in \cM_{g-3, 2}$ and consider another clutching map
$$j:\mm_{3, 1}\rightarrow \mm_{g, 1}, \ \mbox{  } \ j\bigl([B, p]\bigr):=[B\cup _{p\sim p'} C, \ q]. $$
The following formulas are easy to prove, see for instance the proof of Lemma 4.3 in \cite{EH2}  or Lemma 3.3 in \cite{AC}:
$$j^*(\lambda)=\lambda, \ j^*(\psi)=0, \ j^*(\delta_0)=0, \ j^*(\delta_{g-3})=-\psi, \ j^*(\delta_{g-2})=\delta_1, \ j^*(\delta_{g-1})=\delta_2,$$
and $j^*(\delta_i)=0$, for $i=1, \ldots, g-4$.

\vskip 3pt

The next result establishes Theorem \ref{nonexistence} for $g=3$. In genus $2$, it is shown in \cite{EH2} Lemma 3.3 that the inequality $\rho(2, r, d, \alpha^{\ell}(p))\geq -1$ holds for any pointed curve $[C, p]\in \cM_{2,1}$ and any linear series $\ell\in G^r_d(C)$, with equality if and only if $p\in C$ is a Weierstrass point and $\ell=|(r+2)\cdot p|+(d-r-2)\cdot p$. This establishes Theorem \ref{nonexistence} in genus $2$. The genus $3$ analogue of this statement will be the starting step in our induction argument.

\begin{proposition}\label{gen3}
Let $[C, p]\in \cM_{3, 1}$ be a pointed curve of genus $3$ and $\ell\in G^r_d(C)$. If $\rho(3, r, d, \alpha^{\ell}(p))\leq -2$, then either $p\in C$ is a hyperflex, that is, $K_C=\OO_C(4p)$, or else, $C$ is hyperelliptic and $p\in C$ is a Weierstrass point. In particular, for any Schubert index $\bar{\alpha}:\alpha_0\leq \ldots \leq \alpha_{r}\leq d-r$ such that $\rho(3, r, d, \bar{\alpha})\leq -2$, each component of the locus
$$\Bigl\{[C, p]\in \cM_{3, 1}: G^r_d(C, (p, \bar{\alpha}))\neq \emptyset\Bigr\}$$
has codimension at least $2$.
\end{proposition}

\begin{proof} From Riemann-Roch it follows that if $\Lambda\in G^r_d(C)$ is a linear series with $r\geq 1$, then the inequality $d\geq 3+r$ always holds except
when $r=2$ and $\Lambda=|K_C|$, or when $r=1$. Then $d=3$ and $\Lambda=|K_C(-p)|\in W^1_3(C)$ when $C$ is non-hyperelliptic, or $d\in \{2, 3\}$ when $C$ is hyperelliptic respectively.

\vskip 3pt

We fix now a linear series $\ell\in G^r_d(C)$ and let $a_i:=a_i^{\ell}(p)=\alpha_i^{\ell}(p)+i$ be the $i$-th entry in the vanishing sequence of $\ell$ at $p$. By definition, then
$\ell(-a_i\cdot p)\in G_{d-a_i}^{r-i}(C)$, for each $i=0, \ldots, r$. Assume first that $C$ is non-hyperelliptic. Applying the previous observation, we write the following inequalities:
\begin{equation}\label{ineq0}
d-a_{r-i}\geq i+3\  \mbox{ for }i=0, \ldots, r-3,\  \ \ \mbox{  } \mbox{ and }
\end{equation}

\begin{equation}\label{ineq}
d-a_{r-2}\geq 4, \ \ d-a_{r-1}\geq 3, \ \ d-a_r\geq 0.
\end{equation}
Adding these inequalities up, we obtain that $(r+1)(d-r)-\sum_{i=0}^r \alpha_i^{\ell}(p)\geq 3r-2$, or equivalently,
$\rho(3, r, d, \alpha^{\ell}(p))\geq -2$. Equality holds only if the three inequalities appearing in (\ref{ineq}) are all equalities. When this happens,  then  $(a_{r-2}, a_{r-1}, a_r)=(d-4, d-3, d)$ and we find that $h^0(C, \OO_C(4p))=3$, that is, $p\in C$ is a hyperflex. This condition defines a codimension $2$ subvariety of $\mm_{3, 1}$. Thus, when $p\in C$ is not a hyperflex, then $\rho(g, r, d, \alpha^{\ell}(p))\geq -1$.

Assume now that $C$ is hyperelliptic. The inequalities in (\ref{ineq0}) still hold, as well as, $d-a_{r-2}\geq 4$ and $d-a_r\geq 0$.
Assuming that $a_{r-1}=2$, one finds that $h^0(C, \OO_C(2p))=2$, that is, $p\in C$ is a Weierstrass point.
\end{proof}

\vskip 3pt

\noindent \emph{Proof of Theorem \ref{nonexistence}}. We fix a Schubert index $\bar{\alpha}:\alpha_0\leq \ldots \leq \alpha_i\leq d-r$, such that
$\rho(g, r, d, \bar{\alpha})<-1$, and assume that for a general curve $C$, there exist a point $p\in C$ with $G^r_d(C, p, \bar{\alpha})\neq \emptyset$. The locus $\cD:=\{[C, p]\in \cM_{g, 1}: G^r_d(C, p, \bar{\alpha})\neq \emptyset\}$ is then a divisor in $\cM_{g, 1}$ (by using the result of Eisenbud and Harris \cite{EH1} mentioned in the introduction, the possibility $\cD=\cM_{g, 1}$ can be ruled out). We express the class of its closure $\dd$ in $\mm_{g, 1}$ in terms of the generators of $\mbox{Pic}(\mm_{g, 1})$:
$$[\dd]=a\lambda+c\psi-b_{\mathrm{irr}}\delta_{\mathrm{irr}}-\sum_{i=1}^{g-1} b_i\delta_i.$$
 A flag curve $X$ having a rational spine consisting of a tree of $\PP^1$'s and $g$ elliptic tails satisfies the \emph{pointed} Brill-Noether theorem with respect to marked points lying on the rational spine, see \cite{EH1} Theorem 1.1; precisely, if $p_1, \ldots, p_n$ are smooth points of $X$ not lying on any of the elliptic tails, then $\rho(g, r, d, \alpha^{\ell}(p_1), \ldots, \alpha^{\ell}(p_n))\geq 0$, for any limit linear series $\ell$ of type $\mathfrak g^r_d$ on $X$.  It thus follows that $\mbox{Im}(\varphi_{g, 1})\cap \dd=\emptyset$, in particular
$\varphi^*_{g, 1}([\dd])=0$. Using the description in \cite{EH2} Lemma 4.2 for the pull-back map at the level of divisors
$\varphi_{g, 1}^*:\mbox{Pic}(\mm_{g, 1})\rightarrow \mbox{Pic}(\mm_{0, g+1})$, we obtain that for $i=1, \ldots, g-2$ the following  relations hold:

\begin{equation}\label{bi}
b_i=\frac{(g-i)(g-i-1)}{g(g-1)}c+\frac{i(g-i)}{g-1} b_{g-1}.
\end{equation}

Next we employ Proposition \ref{gen3}. From the additivity of the Brill-Noether number, we find that the pointed curve $[B\cup_p C, \ q]\in \mm_{g, 1}$ carries  no limit linear series $\ell\in G^r_d(C)$ with $\rho(g, r, d, \alpha^{\ell}(q))<-1$, therefore,  $[j^*(\dd)]=0$, where we recall that $j:\mm_{3, 1}\rightarrow \mm_{g, 1}$ was the clutching map attaching a fixed pointed curve of genus $g-3$. In terms of coefficients of $[\dd]$ this information is expressed as follows:
$$0=[j^*(\dd)]=a\lambda-b_{\mathrm{irr}} \delta_{\mathrm{irr}}+b_{g-3} \psi-b_{g-2} \delta_1-b_{g-1}\delta_2\in \mbox{Pic}(\mm_{3, 1}).$$
From the independence of the boundary divisor classes on $\mm_{3, 1}$, we find that $a=b_{\mathrm{irr}}=b_{g-1}=b_{g-2}=b_{g-3}=0$. Using repeatedly (\ref{bi}), we find that $[\dd]=0\in \mbox{Pic}(\mm_{g, 1})$, that is, $\cD=0$. The contradiction comes from the assumption that $\cD$ was of codimension one in $\cM_{g, 1}$.
\hfill $\Box$

\begin{remark} The argument above is quite close to that in \cite{EH2}, where instead of the map $j$,  the clutching map $\iota:\mm_{2,1}\rightarrow \mm_{g, 1}$ obtained by attaching a fixed $2$-pointed curve of genus $g-2$ at the marked point of each stable curve of genus $2$ is considered. It is showed in \cite{EH2} Lemma 4.3 that $[\iota^*(\dd)]=0$. However, this does not appear to be enough in order to conclude that $[\dd]=0$ and finish the proof. Indeed, in view of Mumford's relation $10\lambda=\delta_{\mathrm{irr}}+2\delta_1$ on $\mm_{2, 1}$, there exists a one-dimensional vector space of divisor classes $\gamma\in \mbox{Pic}(\mm_{g, 1})$ satisfying both constraints $\varphi^*(\gamma)=0$ and $\iota^*(\gamma)=0$. To circumvent this problem, we considered instead the map $j:\mm_{3, 1}\rightarrow \mm_{g, 1}$. This also explains why the induction in the proof of Theorem \ref{nonexistence} starts only from genus $3$.
\end{remark}

\vskip 2pt

We split the proof of Theorem \ref{multiple} in two parts.

\vskip 4pt

\noindent

\noindent \emph{Proof of Theorem \ref{multiple}, the case $n=2$}.  Let $\bar{\alpha}^1$ and $\bar{\alpha}^2$ be Schubert indices of type
$(r, d)$ such that $\rho(g, r, d, \bar{\alpha}^1, \bar{\alpha}^2)<-1$ and assume by contradiction that the locus
$$D:=\bigl\{[C, p_1, p_2]\in \cM_{g, 2}: G^r_d\bigl(C, (p_1, \bar{\alpha}^1), (p_2, \bar{\alpha}^2)\bigr)\neq \emptyset \bigr\},$$
is an effective divisor. We express the class of its closure in $\mm_{g, 2}$ as a combination
$$[\overline{D}]=a\lambda+c_1\psi_1+c_2\psi_2-b_{\mathrm{irr}}\delta_{\mathrm{irr}}-\sum_{i=1}^{g-1} b_{i:1}\ \delta_{i:1}-\sum_{i=0}^{g-1}b_{i:12}\ \delta_{i:12}\in \mbox{Pic}(\mm_{g, 2}).$$
Let $\pi_2:\mm_{g, 2}\rightarrow \mm_{g, 1}$ be the morphism forgetting the second marked point. We claim that the divisor
$(\pi_2)_*(\overline{D}\cdot \Delta_{0:12})$ is trivial. Indeed, suppose this is not the case and we choose a general point $[C, p]\in (\pi_2)_*(\overline{D}\cdot \Delta_{0:12})$. We insert a smooth rational curve denoted by $\PP^1$ at the point $p$ and view $p_1, p_2$ as distinct smooth points on $\PP^1$. Since $[C\cup_p \PP^1, p_1, p_2]\in \overline{D}$, there exists a limit linear series
$\ell=\{\ell_C, \ell_{\PP^1}\}$ on $C\cup_p \PP^1$ with ramification $\alpha^{\ell}(p_1)\geq \bar{\alpha}^1$ and $\alpha^{\ell}(p_2)\geq \bar{\alpha}^2$. We write the following series of inequalities:
$$-1>\rho(g, r, d, \bar{\alpha}^1, \bar{\alpha}^2)\geq \rho(g, r, d, \alpha^{\ell}(p_1), \alpha^{\ell}(p_2))\geq $$
$$\rho(g, r, d, \alpha^{\ell_C}(p))+\rho(0, r, d, \alpha^{\ell_{\PP^1}}(p_1), \alpha^{\ell}(p_1), \alpha^{\ell}(p_2))\geq \rho(g, r, d, \alpha^{\ell_C}(p)),$$
where the third inequality reflects the additivity of the Brill-Noether number that is incorporated in the definition of the limit linear series
(see also \cite{EH1} p. 365), whereas the last inequality follows from the Pl\"ucker formula (see also \cite{EH1} Theorem 1.1).
Thus $\rho(g, r, d, \alpha^{\ell_C}(p))<-1$. Theorem \ref{nonexistence} guarantees that the locus of curves $[C, p]\in \cM_{g, 1}$ satisfying such a Brill-Noether condition is a subvariety of $\cM_{g, 1}$ of codimension at least $2$. This leads to a contradiction since $[C, p]$ was chosen as being a general point of a divisor on $\mm_{g, 1}$, therefore we conclude $(\pi_2)_*(\overline{D}\cdot \Delta_{0:12})=0$. At the level of classes, from the formulas
$$(\pi_2)_*(\psi_1\cdot \delta_{0:12})=(\pi_2)_*(\psi_2\cdot \delta_{0:12})=0, \ (\pi_2)_*(\delta_{i:1}\cdot \delta_{0:12})=0 \ \mbox{ for } \ i=1, \ldots, g-1,$$ respectively $$(\pi_2)_*(\lambda\cdot \delta_{0:12})=\lambda, \ (\pi_2)_*(\delta_{\mathrm{irr}}\cdot \delta_{0:12})=\delta_{\mathrm{irr}}, \ \ (\pi_2)_*(\delta_{0:12}^2)=-\psi_1, \
(\pi_2)_*(\delta_{i:12}\cdot \delta_{0:12})=\delta_{i:1},$$ we obtain the following identity
$$0=(\pi_2)_*([\overline{D}]\cdot \delta_{0:12})=a\lambda-b_{\mathrm{irr}}\delta_{\mathrm{irr}}+b_{0:12}\psi_1-\sum_{i=1}^{g-1} b_{i:12}\delta_{i:1}\in \mathrm{Pic}(\mm_{g, 1}),$$
hence $a=b_{\mathrm{irr}}=b_{i:12}=0$, for $i=0, \ldots, g-1$.

\vskip 3pt

To show that the coefficients $c_1$ and $c_2$ of $[\overline{D}]$ vanish, we note that $\mbox{Im}(\varphi_{g, 2})\cap \overline{D}=\emptyset$
(cf. \cite{EH1} Theorem 1.1), hence
$\varphi_{g, 2}^*([\overline{D}])=0$. We express this pull-back as an element of the Picard group $\mbox{Pic}(\mm_{0, g+2})^{\mathfrak{S}_g}$, where
the symmetric group $\mathfrak{S}_g$ acts on $\mm_{0, g+2}$ by permuting the first $g$ marked points of each element $[R, x_1, \ldots, x_g, p_1, p_2]$. Using for instance \cite{AC}, we write the following formulas:
$$\varphi_{g, 2}^*(\psi_1)=\psi_{p_1}, \ \ \varphi_{g, 2}^*(\psi_2)=\psi_{p_2},\  \  \varphi_{g, 2}^*(\delta_{i:1})=B^{p_1}_i.$$

Via the description of $\mbox{Pic}(\mm_{0, g+2})^{\mathfrak{S}_g}$ in terms of $\mathfrak{S}_g$-invariant boundary classes as well as the expression of the cotangent classes $\psi_{p_1}$ and $\psi_{p_2}$ given in Lemma \ref{invpic},  it is straightforward to check that the divisor classes
$\psi_{p_1}, \psi_{p_2}$ and $\{B^{p_1}_i\}_{i=1}^{g-1}$ are linearly independent. In particular, from the assumption $\varphi_{g, 2}^*(\overline{D})=0$, we conclude that $c_1=c_2=0$. Thus we can apply Lemma \ref{zerodiv}, which establishes Theorem \ref{multiple} for the case of $2$ marked points. \hfill $\Box$

\vskip 3pt

The case $n\geq 3$ can be easily reduced to the situation discussed above:

\vskip 2pt

\noindent \emph{Proof of Theorem \ref{multiple}, the case $n\geq 3$.} We assume, by contradiction, that the locus in moduli
$$D:=\Bigl\{[C, p_1, \ldots, p_n]\in \cM_{g, n}: G^r_d(C, (p_j, \bar{\alpha}^j))\neq \emptyset\Bigr\}$$
is a divisor. We shall show that the $\lambda, \{\psi_j\}_{j=1}^n$ and $\{\delta_{0:S}\}_{|S|=2}$-coefficients of $[\overline{D}]$ are zero, then use Lemma \ref{zerodiv} to conclude. The inductive hypothesis coupled with \cite{EH1} Theorem 1.1 ensures that $D$ enjoys two geometric properties:
\vskip 3pt

\noindent (1) If $\pi_i:\mm_{g, n}\rightarrow \mm_{g, n-1}$ is the map forgetting the $i$-th marked point, then for each $1\leq i\leq n$ and  $j\in \{i\}^c$, we have that $(\pi_i)_*([\overline{D}]\cdot \delta_{0:ij})=0$.

\vskip 2pt

\noindent (2) If $\varphi_{g, n}: \mm_{0, g+n}\rightarrow \mm_{g, n}$ is the flag map, then $\varphi_{g, n}^*([\overline{D}])=0$.

\vskip 4pt

We claim that these restrictions, together with the assumption that $D$ is effective imply that $D=0$. Firstly, condition (1) implies that the $\lambda, \delta_{\mathrm{irr}}, \psi_1, \ldots, \psi_n$-coefficients of $[\overline{D}]$ are equal to zero. Indeed, the $\lambda, \delta_{\mathrm{irr}}$ and $\psi_k$-coefficients of the divisor classes  $[\overline{D}]\in \mbox{Pic}(\mm_{g, n})$ and $(\pi_i)_*([\overline{D}]\cdot \delta_{0:ij})\in \mbox{Pic}(\mm_{g, n-1})$ respectively, are equal for all $k\in \{i, j\}^c$ (this is where the assumption $n\geq 3$ plays a role). But then via condition (2), the $\delta_{0:ij}$-coefficient of $[\overline{D}]$ is also equal to zero, since
$\varphi_{g, n}^*(\delta_{0:ij})=\delta_{0:p_i p_j}$. Furthermore, since all the $\psi$-coefficients in the expression $[\overline{D}]$ are zero, the coefficient of $\delta_{0:p_i p_j}$ in $[\varphi_{g, n}^*(\overline{D})]$ equals the $\delta_{0:ij}$-coefficient in $[\overline{D}]$, thus showing that the latter coefficient is equal to zero. We now apply Lemma \ref{zerodiv} and conclude. \hfill $\Box$

\end{document}